\documentclass[oneside,english]{amsart}
\usepackage[T1]{fontenc}
\usepackage[latin9]{inputenc}
\usepackage{amsthm}
\usepackage{amstext}
\usepackage{amssymb}
\usepackage{esint}
\usepackage{graphicx}
\usepackage{enumerate}

\usepackage{graphicx}
\graphicspath{{noiseimages/}}

\makeatletter
\newtheorem{theorem}{Theorem}[section]

\newtheorem{corollary}[theorem]{Corollary}
\numberwithin{figure}{section}
\theoremstyle{definition}

\theoremstyle{remark}
\newtheorem{remark}[theorem]{Remark}

\numberwithin{equation}{section}
\makeatother

\usepackage{babel}

	\DeclareMathOperator{\loc}{loc}
	
	\DeclareMathOperator*{\esssup}{ess\,sup}

\begin{document}

\title[Spectral Properties]{On Variations of Neumann Eigenvalues of $p$-Laplacian Generated by Measure Preserving Quasiconformal Mappings}

\author{V.~Pchelintsev}

\begin{abstract}
In this paper we study variations of the first non-trivial eigenvalues of the two-dimensional $p$-Laplace operator, $p>2$, generated by measure preserving quasiconformal mappings $\varphi : \mathbb D\to\Omega$, $\Omega \subset\mathbb R^2$. This study is based on the geometric theory of composition operators on Sobolev spaces with applications to sharp embedding theorems. By using a sharp version of the reverse H\"older inequality we obtain lower
estimates of the first non-trivial  eigenvalues for Ahlfors type domains.
\end{abstract}

\maketitle

\section{Introduction}

In the present article we consider the Neumann eigenvalue problem for the two-dimensional degenerate $p$-Laplace operator ($p>2$)
\[
\Delta_p u=\textrm{div}(|\nabla u|^{p-2}\nabla u).
\]

This operator arises in study of vibrations of nonelastic membranes \cite{ AM74, Pol15}. The weak statement of the frequencies problem
for the vibrations of a nonelastic membrane is as follows: a function $u$ solves this previous problem iff
$u \in W^{1}_{p}(\Omega)$ and
$$
\int\limits _{\Omega} (|\nabla u(x)|^{p-2}\nabla u(x) \cdot \nabla v(x))\,dx =
\mu_p \int\limits _{\Omega} |u(x)|^{p-2} u(x) v(x)\,dx
$$
for all $v \in W^{1}_{p}(\Omega)$.

It is known that exact calculations of Neumann eigenvalues are possible in a limited number of cases. So, estimates of Neumann eigenvalues are significant in the spectral theory of elliptic operators.

Lower estimates of the first non-trivial Neumann eigenvalue of the $p$-Laplace operator, $p>2$, are known for convex domains
$\Omega\subset\mathbb R^n$  \cite{ENT}:
$$
\mu_p(\Omega) \geq \left(\frac{\pi_p}{d(\Omega)}\right)^p ,
$$
where $d(\Omega)$ is a diameter of a convex domain $\Omega$, $\pi_p=2 \pi {(p-1)^{\frac{1}{p}}}/({p \sin(\pi/p))}$.

Unfortunately in  non-convex domains $\mu_p(\Omega)$ can not be characterized in the terms of Euclidean diameters. This can be seen by considering a domain consisting of two identical squares connected by a thin corridor \cite{BCDL16}.

The method which allows to obtain estimates of Neumann eigenvalues in non-convex domains was suggested in \cite{GU16}. This method is based on the geometric theory of compositions operators on Sobolev spaces  \cite{GG94,GU10,U93,VU02}. In a series of works \cite{GHPU20,GPU18_1,GPU17_2,GPU19,GU2016}, using this method were obtained lower estimates for the first non-trivial Neumann eigenvalues of the $p$-Laplace ($p>1$) operator in a large class of non-convex domains in terms of (quasi)conformal geometry.

In \cite{GPU2018} were obtained lower estimates of the first non-trivial Neumann eigenvalue of the degenerate $p$-Laplace operator, $p>2$, in a large class of planar domains in terms of the conformal radii of domains.

The aim of this paper is to refine results from \cite{GHPU20} in the case of measure preserving quasiconformal mappings. Namely, we obtain lower estimates of the first non-trivial Neumann eigenvalue of the degenerate $p$-Laplace operator in quasiconformal regular domains generated by measure preserving quasiconformal mappings. This method is based on the geometric theory of compositions operators on Sobolev spaces and the quasiconformal mapping theory (see, for example, \cite{GG94,GU10,U93,VU02}).

Let $\Omega \subset \mathbb R^2$ be a simply connected domain. Then $\Omega$ is called a $K$-quasi\-con\-for\-mal $\beta$-regular domain if there exists a $K$-quasiconformal mapping $\varphi : \mathbb D \to \Omega$ such that
$$
\| J(\cdot, \varphi) \mid L_{\beta}(\Omega)\| < \infty \quad\text{for some}\quad \beta >1,
$$
where $J(z,\varphi)$ is a Jacobian of a $K$-quasiconformal mapping $\varphi : \mathbb D \to \Omega$, see \cite{GPU19}.
The domain $\Omega \subset \mathbb{R}^2$ is called a quasiconformal regular domain if it is a $K$-quasiconformal $\beta$-regular domain for some $\beta>1$.

Note that the class of quasiconformal regular domains includes the class of Gehring domains \cite{AK} and can be described in terms of quasihyperbolic geometry \cite{GM,H1,KOT}.

Let $\varphi:\mathbb D \to \Omega$ be $K$-quasiconformal mappings. We note that there exist so-called measure preserving maps, i.e. $|J(z,\varphi)|=1$, $z \in \mathbb D$. Some examples of such maps can be found in Section 3.
Note that in the class of conformal mappings there exists a unique mapping with the Jacobian is equal to one. This is an identical mapping, i.e. $\varphi(z)=z$, $z \in \mathbb D$.

In the present article we prove that if $\Omega$ is a $K$-quasiconformal $\beta$-regular domain generated by a measure preserving $K$-quasiconformal mapping $\varphi:\mathbb D \to \Omega$, then for $r=p\beta/(\beta-1)$, $p>2$:
\begin{equation*}
\frac{1}{\mu_p(\Omega)}
\leq \\
\inf\limits_{q \in (q^{\ast}, 2]} \left\{B^p_{r,q}(\mathbb D) \pi^{\frac{p}{q}} \right\} K^{\frac{p}{2}} \pi^{\frac{1}{\beta}-1},
\end{equation*}
where $q^{\ast}=2\beta p/(\beta p + 2(\beta-1))$.

Here $B_{r,q}(\mathbb D)$ is the best constant in the (non-weighted) $(r,q)$-Poincar\'e-Sobolev inequality in the unit disc
$\mathbb D \subset \mathbb R^2$ with the upper estimate (see, for example, \cite{GU2016}):
\[
B_{r,q}(\mathbb D) \leq 2^{-\delta} \left(\frac{1- \delta}{1-2 \delta}\right)^{1- \delta}
\pi^{\delta}, \quad \delta = \frac{1}{q}-\frac{1}{r}.
\]

Also we obtain lower estimates of the first non-trivial Neumann eigenvalue of the degenerate $p$-Laplace operator for Ahlfors type domains (i.e. quasidiscs). The suggested approach is based on sharp estimates of the constant in the reverse H\"older inequality for Jacobians of quasiconformal mappings \cite{GPU17_2, GPU2018}.

Recall that $K$-quasidiscs are images of the unit disc under $K$-quasiconformal homeomorphisms of the plane $\mathbb R^2$. This class includes all Lipschitz simply connected domains but also includes a class of fractal domains
(for example, the von Koch snowflake). The Hausdorff dimension of
the quasidisc's boundary can be any number in $[1, 2)$.

Let $\Omega$ be a $K$-quasidisc. Then
\begin{equation*}
\mu_p(\Omega) \geq \frac{M_p(K)}{|\Omega|^{\frac{p}{2}}}=\frac{M_p^{*}(K)}{R^p_{*}},
\end{equation*}
where $R_{*}$ is a radius of a disc $\Omega^{*}$ of the same area as $\Omega$ and
$M_p^{*}(K)=M_p(K)\pi^{-p/2}$ depends only on $p$ and a quasiconformality
coefficient $K$ of $\Omega$. The exact value of $M_p(K)$ is given in Theorem~\ref{Quasidisk}.

\section{Sobolev spaces and quasiconformal mappings}

In this section we recall definitions of Lebesgue and Sobolev spaces and also basic facts about composition operators on Lebesgue and Sobolev spaces and quasiconformal mapping theory. Let $\Omega\subset\mathbb R^n$, $n\geq 2$, be a $n$-dimensional Euclidean domain.
The Lebesgue space $L_p(\Omega)$, $1\leq p<\infty$,
is the space of all locally integrable functions with the finite norm:
\[
\|f\mid L_p(\Omega)\|=\biggr(\int\limits _{\Omega}|f(x)|^{p}\, dx\biggr)^{\frac{1}{p}}<\infty.
\]

The following theorem about composition operators on Lebesgue spaces is well known (see, for example \cite{VU02}):
\begin{theorem}
Let $\varphi :\Omega \to {\Omega'}$ be a weakly differentiable homeomorphism between two domains $\Omega$ and $\Omega'$.
Then the composition operator
\[
\varphi^{*}: L_r(\Omega') \to L_s(\Omega),\,\,\,1 \leq s \leq r< \infty,
\]
is bounded, if and only if $\varphi^{-1}$ possesses the Luzin $N$-property and
\[
\biggr(\int\limits _{\Omega'}\left|J(y,\varphi^{-1})\right|^{\frac{r}{r-s}}\, dy\biggr)^{\frac{r-s}{rs}}=K< \infty,\,\,\,1 \leq s<r< \infty,
\]
\[
\esssup\limits_{y \in \Omega'}\left|J(y,\varphi^{-1})\right|^{\frac{1}{s}}=K< \infty,\,\,\,1 \leq s=r< \infty.
\]
The norm of the composition operator $\|\varphi^{*}\|=K$.
\end{theorem}

The Sobolev space $W^1_p(\Omega)$, $1\leq p\leq\infty$ is defined
as a Banach space of locally integrable weakly differentiable functions
$f:\Omega\to\mathbb{R}$ equipped with the following norm:
\[
\|f\mid W^1_p(\Omega)\|=\biggr(\int\limits _{\Omega}|f(x)|^{p}\, dx\biggr)^{\frac{1}{p}}+
\biggr(\int\limits _{\Omega}|\nabla f(x)|^{p}\, dx\biggr)^{\frac{1}{p}},
\]
where $\nabla f$ is the weak gradient of the function $f$. Recall that the Sobolev space $W^1_p(\Omega)$ coincides with the closure of the space of smooth functions $C^{\infty}(\Omega)$ in the norm of $W^1_p(\Omega)$.

The homogeneous seminormed Sobolev space $L^1_p(\Omega)$, $1\leq p<\infty$,
is the space of all locally integrable weakly differentiable functions equipped
with the following seminorm:
\[
\|f\mid L^1_p(\Omega)\|=\biggr(\int\limits _{\Omega}|\nabla f(x)|^{p}\, dx\biggr)^{\frac{1}{p}}.
\]

We consider the Sobolev spaces as Banach spaces of equivalence classes of functions up to a set of $p$-capacity zero \cite{M}.

Let $\Omega\subset\mathbb R^n$ be an open set. A mapping $\varphi:\Omega\to\mathbb R^n$ belongs to $L^1_{p,\loc}(\Omega)$,
$1\leq p\leq\infty$, if its coordinate functions $\varphi_j$ belong to $L^1_{p,\loc}(\Omega)$, $j=1,\dots,n$.
In this case the formal Jacobi matrix
$D\varphi(x)=\left(\frac{\partial \varphi_i}{\partial x_j}(x)\right)$, $i,j=1,\dots,n$,
and its determinant (Jacobian) $J(x,\varphi)=\det D\varphi(x)$ are well defined at
almost all points $x\in \Omega$. The norm $|D\varphi(x)|$ of the matrix
$D\varphi(x)$ is the norm of the corresponding linear operator $D\varphi (x):\mathbb R^n \rightarrow \mathbb R^n$ defined by the matrix $D\varphi(x)$.

Let $\varphi:\Omega\to\widetilde{\Omega}$ be weakly differentiable in $\Omega$. The mapping $\varphi$ is the mapping of finite distortion if $|D\varphi(z)|=0$ for almost all $x\in Z=\{z\in\Omega : J(z,\varphi)=0\}$.

A mapping $\varphi:\Omega\to\mathbb R^n$ possesses the Luzin $N$-property if a image of any set of measure zero has measure zero.
Mote that any Lipschitz mapping possesses the Luzin $N$-property.

The following theorem gives the analytic description of composition operators on Sobolev spaces:

\begin{theorem}
\label{CompTh} \cite{U93,VU02} A homeomorphism $\varphi:\Omega\to\Omega'$
between two domains $\Omega$ and $\Omega'$ induces a bounded composition
operator
\[
\varphi^{\ast}:L^1_p(\Omega')\to L^1_q(\Omega),\,\,\,1\leq q< p<\infty,
\]
 if and only if $\varphi\in W_{1,\loc}^{1}(\Omega)$, has finite distortion,
and
$$
K_{p,q}(\Omega)=\left(\int\limits_\Omega \left(\frac{|D\varphi(x)|^p}{|J(x,\varphi)|}\right)^\frac{q}{p-q}~dx\right)^\frac{p-q}{pq}<\infty.
$$
\end{theorem}

Recall that a homeomorphism $\varphi: \Omega\to \Omega'$ is called a $K$-quasiconformal mapping if $\varphi\in W^1_{n,\loc}(\Omega)$ and there exists a constant $1\leq K<\infty$ such that
$$
|D\varphi(x)|^n\leq K |J(x,\varphi)|\,\,\text{for almost all}\,\,x\in\Omega.
$$

Quasiconformal mappings have a finite distortion, i.~e.  $D\varphi(x)=0$ for almost all points $x$
that belongs to set $Z=\{x\in \Omega:J(x,\varphi)=0\}$ and any quasiconformal mapping possesses Luzin $N$-property.  A mapping which is inverse to a quasiconformal mapping is also quasiconformal.

If $\varphi : \Omega \to \Omega'$ is a $K$-quasiconformal mapping then $\varphi$ is differentiable almost everywhere in $\Omega$ and
$$
|J(x,\varphi)|=J_{\varphi}(x):=\lim\limits_{r\to 0}\frac{|\varphi(B(x,r))|}{|B(x,r)|}\,\,\text{for almost all}\,\,x\in\Omega.
$$

If $K\equiv 1$ then $1$-quasiconformal homeomorphisms are conformal mappings and in the space $\mathbb R^n$, $n\geq 3$, are exhausted by M\"obius transformations.

\section{Eigenvalue Problem for Neumann $p$-Laplacian}

In \cite{GHPU20} were obtained lower estimates for the first non-trivial eigenvalues of the degenerate $p$-Laplace operator ($p>2$) with the Neumann boundary condition in quasiconformal regular domains.
\begin{theorem}
\label{thm:estball}
Let $\Omega$ be a $K$-quasiconformal $\beta$-regular domain, $r=p\beta/(\beta-1)$, $p>2$.
Then the following inequality holds
\begin{equation*}
\frac{1}{\mu_p(\Omega)}
\leq \\
\inf\limits_{q \in (q^{\ast}, 2]} \left\{2^p\left(\frac{1-\frac{1}{q}+\frac{1}{r}}{\frac{1}{2}-\frac{1}{q}+\frac{1}{r}}\right)^{p-\frac{p}{q}+\frac{p}{r}} \pi^{\frac{p}{r}-\frac{p}{2}}\right\} K^{\frac{p}{2}} |\Omega|^{\frac{p-2}{2}} \cdot ||J_{\varphi}\,|\,L_{\beta}(\mathbb D)||,
\end{equation*}
where $q^{\ast}=2\beta p/(\beta p + 2(\beta-1))$.
\end{theorem}

In the case of $K$-quasiconformal $\infty$-regular domains, we have the following result \cite{GHPU20}:

\begin{theorem}
\label{thm:est2}
Let $\Omega$ be a $K$-quasiconformal $\infty$-regular domain.
Then for any $p>2$ the following inequality holds
$$
\frac{1}{\mu_p(\Omega)}
 \leq  \inf\limits_{q \in (q^{\ast}, 2]} \left\{2^p\left(\frac{1-\frac{1}{q}+\frac{1}{p}}{\frac{1}{2}-\frac{1}{q}+\frac{1}{p}}\right)^{p+1-\frac{p}{q}} \pi^{1-\frac{p}{2}}\right\} K^{\frac{p}{2}} |\Omega|^{\frac{p-2}{2}} \cdot ||J_{\varphi}\,|\,L_{\infty}(\mathbb D)||,
$$
where $q^{\ast}=2p/(p + 2)$.
\end{theorem}

As examples, we consider the domains bounded by an epicycloid. Since
the domains bounded by an epicycloid are $K$-quasiconformal $\infty$-regular, we
can apply Theorem \ref{thm:est2}, i.e.:

{\bf Example 1.} For $n \in \mathbb N$, the homeomorphism
\[
\varphi(z)=A\left(z+\frac{z^n}{n}\right)+B\left(\overline{z}+\frac{\overline{z}^n}{n}\right), \quad z=x+iy, \quad A>B\geq0,
\]
is quasiconformal with $K=(A+B)/(A-B)$ and maps the unit disc $\mathbb D$ onto the domain $\Omega_n$ bounded by an epicycloid of
$(n - 1)$ cusps, inscribed in the ellipse with semi-axes $(A+B)(n+1)/n$ and $(A-B)(n+1)/n$.

Now we calculate the area of domain $\Omega_n$ and estimate the quantity $||J_{\varphi}\,|\,L_{\infty}(\mathbb D)||$. Straightforward calculations yield
\[
|\Omega_n|=\iint\limits_{\mathbb D} |J(z,\varphi)|~dxdy=\iint\limits_{\mathbb D}(A^2-B^2)|1+z^{n-1}|^2~dxdy
=(A^2-B^2)\frac{n+1}{n}\pi.
\]
\[
||J_{\varphi}\,|\,L_{\infty}(\mathbb D)||
=\esssup\limits_{|z|<1}\left[(A^2-B^2)|1+z^{n-1}|^2\right]
\leq 4(A^2-B^2).
\]

Then by Theorem \ref{thm:est2} we have
$$
\frac{1}{\mu_p(\Omega)}
 \leq  \inf\limits_{q \in (q^{\ast}, 2]} \left\{\left(\frac{1-\frac{1}{q}+\frac{1}{p}}{\frac{1}{2}-\frac{1}{q}+\frac{1}{p}}\right)^{p+1-\frac{p}{q}}\right\}
2^{p+2}(A+B)^p \left(\frac{n+1}{n}\right)^{\frac{p}{2}-1},
$$
where $q^{\ast}=2p/(p + 2)$.

\vskip 0.3cm

Now we construct $K$-quasiconformal measure preserving mappings $\varphi:\mathbb D \to \Omega$, i.e. $|J(z,\varphi)|=1$, $z \in \mathbb D$.

\vskip 0.3cm

{\bf Example 2.} The homeomorphism
\[
\varphi(z)= \sqrt{a^2+1}z+a \overline{z}, \quad z=x+iy, \quad a\geq 0,
\]
is a $K$-quasiconformal with $K=\frac{\sqrt{a^2+1}+a}{\sqrt{a^2+1}-a}$ and maps the unit disc $\mathbb D$ onto the interior of ellipse
$$
\Omega_e= \left\{(x,y) \in \mathbb R^2: \frac{x^2}{(\sqrt{a^2+1}+a)^2}+\frac{y^2}{(\sqrt{a^2+1}-a)^2}=1\right\}.
$$
It is not difficult to verify that the Jacobian $J(z,\varphi)=|\varphi_{z}|^2-|\varphi_{\overline{z}}|^2=1$.

\vskip 0.3cm

{\bf Example 3.} The homeomorphism
\[
\varphi(z)= \sqrt{2}(1+z)^{\frac{3}{4}} (1+\overline{z})^{\frac{1}{4}}, \quad z=x+iy,
\]
is a $K$-quasiconformal with $K=2$ and maps the unit disc $\mathbb D$ onto the interior of the ``rose petal"
$$
\Omega_p:=\left\{(\rho, \theta) \in \mathbb R^2:\rho=2\sqrt{2}\cos(2 \theta), \quad -\frac{\pi}{4} \leq \theta \leq \frac{\pi}{4}\right\}.
$$
It is easy to verify that the Jacobian $J(z,\varphi)=|\varphi_{z}|^2-|\varphi_{\overline{z}}|^2=1$.

\vskip 0.3cm

{\bf Example 4.} Let $f\in L^{\infty}(\mathbb R)$. Then $\varphi(x,y)=(x+f(y),\,y)$ is a quasiconformal mapping with a quasiconformality coefficient $K=\lambda/J_{\varphi}(x,y)$.
Here $\lambda$ is the largest eigenvalue of the matrix $Q=DD^T$, where
$D=D\varphi(x,y)$ is Jacobi matrix of mapping $\varphi=\varphi(x,y)$ and $J_{\varphi}(x,y)=\det D\varphi(x,y)$ is its Jacobian.

It is easy to see that the Jacobi matrix corresponding to the mapping $\varphi=\varphi(x,y)$ has the form
\[
D=\left(\begin{array}{cc}
1 & f'(y)\\
0 & 1
\end{array}\right).
\]

A basic calculation implies $J_{\varphi}(x,y)=1$ and
\[
\lambda=\left(1+\frac{\left(f'(y)\right)^{2}}{2}\right)\left(1+\sqrt{1-\frac{4}{\left(2+\left(f'(y)\right)^{2}\right)^{2}}}\right)\,.
\]

Therefore any mapping $\varphi=\varphi(x,y)$ is a quasiconformal mapping
from $\mathbb R^{2}\to \mathbb R^{2}$ with $J_{\varphi}(x,y)=1$ and arbitrary large quasiconformality coefficient.

We can use their restrictions $\varphi|_{\mathbb D}$ to the unit disc $\mathbb D$. Images
can be very exotic quasidiscs.

If $a>0$ then mappings $\varphi(x,y)=(ax+f(y),\,\frac{1}{a}y)$ have
similar properties.

\vskip 0.3cm

So, the class of measure preserving $K$-quasiconformal mappings is not empty. If $\Omega$ is a $K$-quasiconformal $\beta$-regular domain generated by a measure preserving $K$-quasiconformal mapping $\varphi:\mathbb D \to \Omega$ then we obtain lower estimates of the first non-trivial Neumann eigenvalues
of the degenerate $p$-Laplace operator ($p>2$) via
the Sobolev-Poincar\'e constant for the unit disc $\mathbb D$ and a quasiconformality
coefficient K of $\Omega$. Namely:
\begin{theorem}
Let $\Omega$ be a $K$-quasiconformal $\beta$-regular domain generated by a measure preserving $K$-quasiconformal mapping $\varphi:\mathbb D \to \Omega$, $r=p\beta/(\beta-1)$, $p>2$.
Then the following inequality holds
\begin{equation*}
\frac{1}{\mu_p(\Omega)}
\leq \\
\inf\limits_{q \in (q^{\ast}, 2]} \left\{2^p\left(\frac{1-\frac{1}{q}+\frac{1}{r}}{\frac{1}{2}-\frac{1}{q}+\frac{1}{r}}\right)^{p-\frac{p}{q}+\frac{p}{r}} \right\} K^{\frac{p}{2}},
\end{equation*}
where $q^{\ast}=2\beta p/(\beta p + 2(\beta-1))$.
\end{theorem}

\begin{proof}
According to Theorem \ref{thm:estball} we have
\begin{equation*}
\frac{1}{\mu_p(\Omega)}
\leq \\
\inf\limits_{q \in (q^{\ast}, 2]} \left\{2^p\left(\frac{1-\frac{1}{q}+\frac{1}{r}}{\frac{1}{2}-\frac{1}{q}+\frac{1}{r}}\right)^{p-\frac{p}{q}+\frac{p}{r}} \pi^{\frac{p}{r}-\frac{p}{2}}\right\} K^{\frac{p}{2}} |\Omega|^{\frac{p-2}{2}} \cdot ||J_{\varphi}\,|\,L_{\beta}(\mathbb D)||,
\end{equation*}
where $q^{\ast}=2\beta p/(\beta p + 2(\beta-1))$. Because $\Omega$ is a $K$-quasiconformal $\beta$-regular domain generated by a measure preserving $K$-quasiconformal mapping $\varphi:\mathbb D \to \Omega$ then we have that the quality
$||J_{\varphi}\,|\,L_{\beta}(\mathbb D)||=\pi^{\frac{1}{\beta}}$ and $|\Omega|=|\mathbb D|=\pi$. After some computations we obtain
\begin{equation*}
\frac{1}{\mu_p(\Omega)}
\leq \\
\inf\limits_{q \in (q^{\ast}, 2]} \left\{2^p\left(\frac{1-\frac{1}{q}+\frac{1}{r}}{\frac{1}{2}-\frac{1}{q}+\frac{1}{r}}\right)^{p-\frac{p}{q}+\frac{p}{r}} \right\} K^{\frac{p}{2}},
\end{equation*}
where $q^{\ast}=2\beta p/(\beta p + 2(\beta-1))$.
\end{proof}

In the case of $K$-quasiconformal $\infty$-regular domains generated by a measure preserving $K$-quasiconformal mapping $\varphi: \mathbb D \to \Omega$, we get the following assertion:

\begin{theorem}
\label{Cor-3.4}
Let $\Omega$ be a $K$-quasiconformal $\infty$-regular domain generated by a measure preserving $K$-quasiconformal mapping $\varphi:\mathbb D \to \Omega$.
Then for any $p>2$ the following inequality holds
$$
\frac{1}{\mu_p(\Omega)}
 \leq  \inf\limits_{q \in (q^{\ast}, 2]} \left\{2^p\left(\frac{1-\frac{1}{q}+\frac{1}{p}}{\frac{1}{2}-\frac{1}{q}+\frac{1}{p}}\right)^{p+1-\frac{p}{q}}\right\} K^{\frac{p}{2}},
$$
where $q^{\ast}=2p/(p + 2)$.
\end{theorem}

As an application of Theorem~\ref{Cor-3.4} we obtain lower estimates of the first non-trivial Neumann eigenvalues
of the degenerate $p$-Laplace operator for domains from examples 2 and 3, i.e. $\Omega=\Omega_e$ and $\Omega=\Omega_p$.  In this case we have
$$
\frac{1}{\mu_p(\Omega_e)}
 \leq  \inf\limits_{q \in (q^{\ast}, 2]} \left\{2^p\left(\frac{1-\frac{1}{q}+\frac{1}{p}}{\frac{1}{2}-\frac{1}{q}+\frac{1}{p}}\right)^{p+1-\frac{p}{q}}\right\} \left(\frac{\sqrt{a^2+1}+a}{\sqrt{a^2+1}-a}\right)^{\frac{p}{2}},
$$
and
$$
\frac{1}{\mu_p(\Omega_p)}
 \leq  \inf\limits_{q \in (q^{\ast}, 2]} \left\{\left(\frac{1-\frac{1}{q}+\frac{1}{p}}{\frac{1}{2}-\frac{1}{q}+\frac{1}{p}}\right)^{p+1-\frac{p}{q}}\right\} (2\sqrt{2})^{p},
$$
where $q^{\ast}=2p/(p + 2)$.

\section{Spectral estimates in quasidiscs}
In this section we precise Theorem~\ref{thm:estball} for Ahlfors-type domains (i.e. quasidiscs) using the weak inverse H\"older inequality and the sharp estimates of the constants in doubling conditions for measures generated by Jacobians of quasiconformal mappings  \cite{GPU17_2}.

Recall that a domain $\Omega$ is called a $K$-quasidisc if it is the image of the unit disc $\mathbb D$ under a $K$-quasicon\-for\-mal homeomorphism of the plane onto itself. A domain $\Omega$ is a quasidisc if it is a $K$-quasidisc for some $K \geq 1$.

According to \cite{GH01}, the boundary of any $K$-quasidisc $\Omega$
admits a $K^{2}$-quasi\-con\-for\-mal reflection and thus, for example,
any quasiconformal homeomorphism $\varphi:\mathbb{D}\to\Omega$ can be
extended to a $K^{2}$-quasiconformal homeomorphism of the whole plane
to itself.

Recall that for any planar $K$-quasiconformal homeomorphism $\varphi:\Omega\rightarrow \Omega'$
the following sharp result is known: $J(z,\varphi)\in L^p_{\loc}(\Omega)$
for any $1 \leq p<\frac{K}{K-1}$ (\cite{A94,G81}).

In \cite{GPU17_2} was proved but not formulated the result concerning an estimate of the constant in the inverse H\"older inequality for Jacobians of quasiconformal mappings.

\vskip 0.2cm

\begin{theorem}
\label{thm:IHIN}
Let $\varphi:\mathbb R^2 \to \mathbb R^2$ be a $K$-quasiconformal mapping. Then for every disc $\mathbb D \subset \mathbb R^2$ and
for any $1<\kappa<\frac{K}{K-1}$ the inverse H\"older inequality
\begin{equation*}\label{RHJ}
\left(\iint\limits_{\mathbb D} |J(z,\varphi)|^{\kappa}~dxdy \right)^{\frac{1}{\kappa}}
\leq \frac{C_\kappa^2 K \pi^{\frac{1}{\kappa}-1}}{4}
\exp\left\{{\frac{K \pi^2(2+ \pi^2)^2}{2\log3}}\right\}\iint\limits_{\mathbb D} |J(z,\varphi)|~dxdy
\end{equation*}
holds. Here
$$
C_\kappa=\frac{10^{6}}{[(2\kappa -1)(1- \nu)]^{1/2\kappa}}, \quad \nu = 10^{8 \kappa}\frac{2\kappa -2}{2\kappa -1}(24\pi^2K)^{2\kappa}<1.
$$
\end{theorem}

If $\Omega$ is a $K$-quasidisc, then given the previous theorem and that a quasiconformal mapping $\varphi:\mathbb{D}\to\Omega$ allows $K^2$-quasiconformal reflection \cite{Ahl66, GH01}, we obtain the following assertion.

\begin{corollary}\label{Est_Der}
Let $\Omega\subset\mathbb R^2$ be a $K$-quasidisc and $\varphi:\mathbb D \to \Omega$ be an $K$-quasiconformal mapping. Assume that  $1<\kappa<\frac{K}{K-1}$.
Then
\begin{equation*}\label{Ineq_2}
\left(\iint\limits_{\mathbb D} |J(z,\varphi)|^{\kappa}~dxdy \right)^{\frac{1}{\kappa}}
\leq \frac{C_\kappa^2 K^2 \pi^{\frac{1}{\kappa}-1}}{4}
\exp\left\{{\frac{K^2 \pi^2(2+ \pi^2)^2}{2\log3}}\right\}\cdot |\Omega|.
\end{equation*}
where
$$
C_\kappa=\frac{10^{6}}{[(2\kappa -1)(1- \nu)]^{1/2\kappa}}, \quad \nu = 10^{8 \kappa}\frac{2\kappa -2}{2\kappa -1}(24\pi^2K^2)^{2\kappa}<1.
$$
\end{corollary}

Combining Theorem~\ref{thm:estball} and Corollary~\ref{Est_Der} we obtain
spectral estimates of the degenerate $p$-Laplace operator ($p>2$)
with Neumann boundary conditions in Ahlfors-type domains.

\begin{theorem}\label{Quasidisk}
Let $\Omega$ be a $K$-quasidisc. Then
\begin{equation*}
\mu_p(\Omega) \geq \frac{M_p(K)}{|\Omega|^{\frac{p}{2}}}=\frac{M_p^{*}(K)}{R^p_{*}},
\end{equation*}
where $R_{*}$ is a radius of a disc $\Omega^{*}$ of the same area as $\Omega$ and
$M_p^{*}(K)=M_p(K)\pi^{-p/2}$.
\end{theorem}

The quantity $M_p(K)$ depends only on $p$ and a quasiconformality
coefficient $K$ of $\Omega$:
\begin{multline*}
M(K):= \frac{\pi^{\frac{p}{2}}}{2^{p-2}K^{\frac{p}{2}+2}}
\exp\left\{{-\frac{K^2 \pi^2(2+ \pi^2)^2}{2\log3}}\right\} \\
\times
\inf\limits_{\beta \in (1,\beta^{*})}
\inf\limits_{q \in (q^{\ast}, 2]} \left\{\left(\frac{1-\frac{1}{q}+\frac{1}{r}}{\frac{1}{2}-\frac{1}{q}+\frac{1}{r}}\right)^{-p+\frac{p}{q}-\frac{p}{r}}C_\beta^{-2}\right\},
\end{multline*}

\[
C_\beta=\frac{10^{6}}{[(2\beta -1)(1- \nu(\beta))]^{1/2\beta}},
\]
where $\beta^{*}=\min{\left(\frac{K}{K-1}, \widetilde{\beta}\right)}$, and $\widetilde{\beta}$ is the unique solution of the equation
$$
\nu(\beta):=10^{8 \beta}\frac{2\beta -2}{2\beta -1}(24\pi^2K^2)^{2\beta}=1.
$$
The function $\nu(\beta)$ is a monotone increasing function. Hence for
any $\beta < \beta^{*}$ the number $(1- \nu(\beta))>0$ and $C_\beta > 0$.

\begin{proof}
Given that, for $K\geq 1$, $K$-quasidiscs are $K$-quasiconformal $\beta$-regular domains if $1<\beta<\frac{K}{K-1}$. Therefore, by Theorem~\ref{thm:estball} for $1<\beta<\frac{K}{K-1}$, $r=p\beta/(\beta-1)$ and $p>2$ we have
\begin{equation}\label{Inequal_1}
\frac{1}{\mu_p(\Omega)}
\leq \\
\inf\limits_{q \in (q^{\ast}, 2]} \left\{2^p\left(\frac{1-\frac{1}{q}+\frac{1}{r}}{\frac{1}{2}-\frac{1}{q}+\frac{1}{r}}\right)^{p-\frac{p}{q}+\frac{p}{r}} \pi^{\frac{p}{r}-\frac{p}{2}}\right\} K^{\frac{p}{2}} |\Omega|^{\frac{p-2}{2}} \cdot ||J_{\varphi}\,|\,L_{\beta}(\mathbb D)||,
\end{equation}
where $q^{\ast}=2\beta p/(\beta p + 2(\beta-1))$.
Now, using Corollary~\ref{Est_Der} we estimate the quantity $\|J_{\varphi^{-1}}\,|\,L^{\beta}(\mathbb D)\|$.
Direct calculations yield
\begin{multline}\label{Inequal_2}
\|J_{\varphi}\,|\,L^{\beta}(\mathbb D)\| =
\left(\iint\limits_{\mathbb D} |J(z,\varphi)|^{\beta}~dxdy \right)^{\frac{1}{\beta}} \\
\leq \frac{C^2_{\beta} K^2 \pi^{\frac{1-\beta}{\beta}}}{4} \exp\left\{{\frac{K^2 \pi^2(2+ \pi^2)^2}{2\log3}}\right\} \cdot |\Omega|.
\end{multline}
Finally, combining inequality \eqref{Inequal_1} with inequality \eqref{Inequal_2} after some computations, we obtain
the required inequality.
\end{proof}

\begin{remark}
In the case of conformal mappings lower estimates of the first non-trivial Neumann eigenvalue of the degenerate $p$-Laplace operator in quasidiscs were obtained in \cite{GPU2018}.
\end{remark}

{\bf Acknowledgements.}
The author thanks Vladimir Gol'dshtein and Alexander Ukhlov for useful discussions and valuable comments.
This work was supported by the RSF Grant No. 20-71-00037.

\vskip 0.3cm



\begin{thebibliography}{}
%

\bibitem{Ahl66}
L.~Ahlfors, {\em Lectures on quasiconformal mappings}, D. Van Nostrand Co., Inc., Toronto, Ont.-New York-London (1966).

\bibitem{AK}
K.~Astala, P.~Koskela, ``Quasiconformal mappings and global integrability of the derivative," {\em J. Anal. Math.} {\bf 57}, 203--220 (1991).

\bibitem{A94}
K.~Astala, ``Area distortion of quasiconformal mappings," {\em Acta Math.} {\bf 173}, 37--60 (1994).

\bibitem{AM74}
G.~Astarita, G.~Marrucci, \textit{Principles of non-Newtonian fluids mechanics}, McGraw-Hill, New York,  (1974).

\bibitem{BCDL16}
B.~Brandolini, F.~Chiacchio, E.B.~Dryden, J.J.~Langford, ``Sharp Poincar\'e inequalities in a class of non-covex sets," {\em J. Spectr. Theory} {\bf 8}, 1583--1615 (2018).

\bibitem{ENT}
L.~Esposito, C.~Nitsch, C.~Trombetti, ``Best constants in Poincar\'e inequalities for convex domains",
{\em J. Convex Anal.} {bf 20}, 253--264 (2013).

\bibitem{GM}
F. W.~Gehring, O.~Martio, ``Lipschitz classes and  quasiconformal mappings,"
{\em Ann. Acad. Sci. Fenn. Ser. A I Math}. {\bf 10}, 203--219 (1985).

\bibitem{GH01}
F. W.~Gehring, K.~Hag, ``Reflections on reflections in quasidicks," {\em Report. Univ. Jyv\"askyl\"a} {\bf 83}, 81--90 (2001).

\bibitem{G81}
V.~M.~Gol'dshtein, ``The degree of summability of generalized
derivatives of quasiconformal homeomorphisms," {\em Siberian Math. J.} {\bf 22}, No. 6, 821--836 (1981).

\bibitem{GG94} V.~Gol'dshtein, L.~Gurov, ``Applications of change of variables operators for exact embedding theorems,"
{\em Integral Equ. Oper. Theory} {\bf 19}, 1--24 (1994).

\bibitem{GHPU20}
V.~Gol'dshtein, R.~Hurri-Syrj\"anen V.~Pchelintsev, A.~Ukhlov, ``Space quasiconformal composition operators with applications to Neumann eigenvalues," {\em Anal.Math.Phys.} {\bf 10}, No. 78, https://doi.org/10.1007/s13324-020-00420-0

\bibitem{GPU18_1}
V.~Gol'dshtein, V.~Pchelintsev, A.~Ukhlov, ``Spectral Estimates of the $p$-Laplace Neumann operator and Brennan's Conjecture," {\em Boll. Unione Mat. Ital.} {\bf 11}, 245--264 (2018).

\bibitem{GPU17_2}
V.~Gol'dshtein, V.~Pchelintsev, A.~Ukhlov,
``Integral estimates of conformal derivatives and spectral properties of the Neumann-Laplacian," {\em J. Math. Anal. Appl.} {\bf 463}, 19--39 (2018).

\bibitem{GPU2018}
V.~Gol'dshtein, V.~Pchelintsev, A.~Ukhlov, ``On the first eigenvalue of the degenerate p-Laplace operator in non-convex domains," {\em Integral Equ. Oper. Theory} {\bf 90}, 21 pp (2018).

\bibitem{GPU19}
V.~Gol'dshtein, V.~Pchelintsev, A.~Ukhlov, ``Spectral Properties of the Neumann-Laplace Operator in Quasiconformal Regular Domains," {\em Contemporary Mathematics} {\bf 734}, 129--144 (2019).

\bibitem{GU10}
V.~Gol'dshtein, A.~Ukhlov, ``About homeomorphisms that induce composition operators on Sobolev spaces,"
{\em Complex Var. Elliptic Equ.} {\bf 55}, 833--845 (2010).

\bibitem{GU16} 
V.~Gol'dshtein, A.~Ukhlov, ``On the first Eigenvalues of Free Vibrating Membranes in Conformal Regular Domains,"
{\em Arch. Rational Mech. Anal.} {\bf 221}, 893--915 (2016).

\bibitem{GU2016} 
V.~Gol'dshtein, A.~Ukhlov, ``Spectral estimates of the $p$-Laplace Neumann operator in conformal regular domains,"
{\em Transactions of A. Razmadze Math. Inst.} {\bf 170} No. 1, 137--148 (2016).

\bibitem{H1}
R.~Hurri, ``Poincar\'e domains in $\mathbb R^n$," Ann. Acad. Sci. Fenn., Ser. A, I. Math., Dissertationes, {\bf 71}, 1--42 (1988).

\bibitem{KOT} 
P.~Koskela, J.~Onninen, J.~T.~Tyson,
``Quasihyperbolic boundary conditions and capacity: Poincar\'e domains,"
{\em Math. Ann.} {\bf 323}, 811--830 (2002).

\bibitem{M} 
V.~Maz'ya, {\em Sobolev spaces: with applications to elliptic
partial differential equations}, Springer, Berlin/Heidelberg (2010).

\bibitem{Pol15} 
G.~Poliquin, ``Principal frequency of the p-Laplacian and the inradius of Euclidean domains,"
{\em J. Topol. Anal.} {\bf 7}, 505--511 (2015).

\bibitem{U93} 
A.~Ukhlov, ``On mappings, which induce embeddings of
Sobolev spaces," {\em Siberian Math. J.} {\bf 34}, No. 1, 165--171 (1993).

\bibitem{VU02} 
S.~K.~Vodop'yanov, A.~D.~Ukhlov, ``Superposition operators in Sobolev spaces,"
{\em Russian Mathematics (Izvestiya VUZ. Matematika)} {\ bf 46}, No. 10, 9--31 (2002).

\end{thebibliography}


\vskip 0.3cm

Division for Mathematics and Computer Sciences, Tomsk Polytechnic University,
634050 Tomsk, Lenin Ave. 30, Russia \\
Department of Mathematical Analysis and Theory of Functions,
Tomsk State University, 634050 Tomsk, Lenin Ave. 36, Russia
 						
 \emph{E-mail address:} \email{vpchelintsev@vtomske.ru}   \\

\end{document}